\documentclass[a4paper,10pt]{amsart}
\usepackage{amssymb,amsmath,amsfonts,amsthm}
\usepackage[dvips]{graphicx}
\usepackage{psfrag}
\usepackage{todonotes}
\usepackage{color}

\theoremstyle{plain}
\newtheorem{main}{Theorem}

\newtheorem{theorem}{Theorem}[section]
\newtheorem{lemma}[theorem]{Lemma}
\newtheorem{proposition}[theorem]{Proposition}
\newtheorem{corollary}[theorem]{Corollary}
\theoremstyle{remark}
\newtheorem{remark}[theorem]{Remark}
\newtheorem{definition}[theorem]{Definition}

           \def\ea{\end{array}}
          \def\ec{\end{center}}
     \def\ed{\end{description}}
        \def\ee{\end{equation}}
       \def\eea{\end{eqnarray}}
     \def\eeaa{\end{eqnarray*}}
 \def\et{\end{thebibliography}}

       \def\nt{\noindent}

\def\supp{\operatorname{supp}}

\def\cA{{\mathcal A}}

\def\cD{{\mathcal D}}

\def\cT{{\mathcal T}}

\def\cH{{\mathcal H}}

\def\cF{{\mathcal F}}
\def\cM{{\mathcal M}}

\def\vep{\varepsilon}

\def\TT{{\mathbb T}}
\def\RR{{\mathbb R}}

\title[Decay of correlations]{Decay of correlations for maximal measure
of maps derived from Anosov}
\author{Fan Yang and Jiagang Yang}
\date{\today}
\thanks{F.Y. was partially supported by CAPES; J.Y. was partially supported by CNPq, FAPERJ, and PRONEX.}

\address{Departamento de Geometria, Instituto de Matem\'atica e Estat\'\i stica, Universidade Federal Fluminense, Niter\'oi, Brazil}
\email{yangjg\@@impa.br}

\address{Instituto de Matem\'atica, Universidade Federal do Rio de Janeiro
, Rio de Janeiro, Brazil}
\email{fyang@im.ufrj.br}

\begin{document}

\begin{abstract}
It was proven by Ures that $C^1$ diffeomorphism of $\TT^3$ that is derived from Anosov admits a
unique maximal measure. Here we show that the maximal measure has exponential decay of
correlations for H\"{o}lder observables.

\end{abstract}

\maketitle

\tableofcontents

\setcounter{tocdepth}{1} \tableofcontents

\section{Introduction}
By the early 1970¡¯s, Brin, Pesin~\cite{BP74} and Pugh, Shub~\cite{PSh72} began the study of partially
hyperbolic diffeomorphisms, as an extension of the classical class of Anosov diffeomorphisms.

A diffeomorphism $f$ on a compact manifold $M$ is
\emph{partially hyperbolic} if there is a $df$-invariant splitting of the tangent
bundle $TM = E^s\oplus E^c \oplus E^u$, such that all unit vectors
$v^i\in E^i_x\setminus \{0\}$ ($i= s, c, u$) with $x\in M$
for some suitable Riemannian metric satisfies
$$e^{\lambda_1(x)}\leq |df\mid_{E^s_x}(v^s)|\leq e^{\lambda_2(x)},$$
$$e^{\lambda_3(x)}\leq |df\mid_{E^c_x}(v^c)|\leq e^{\lambda_4(x)},$$
$$e^{\lambda_5(x)}\leq |df\mid_{E^u_x}(v^u)|\leq e^{\lambda_6(x)},$$
where $\lambda_1(x)\leq \lambda_2(x)<\lambda_3(x)\leq\lambda_4(x)<\lambda_5(x)\leq\lambda_6(x)$
and $\lambda_2(x)<0$, $\lambda_5(x)>0$.

We are interested in three dimensional \emph{derived from Anosov diffeomorphisms}, $\cD(A)$, which are
the partially hyperbolic diffeomorphisms in the same isotopy class with some linear Anosov
diffeomorphism $A$. This definition is a generalization of the classical construction of partially hyperbolic, robustly transitive diffeomorphisms by Ma\~{n}\'{e}~\cite{M}. Let us mention that, although in the same isotopy class, the
dynamics of a derived from Anosov diffeomorphism and the linear Anosov one can be quite different. For example,
the center exponent of a volume preserving derived from Anosov diffeomorphism may have different sign with
the center exponent of the linear Anosov diffeomorphism (~\cite{PT}).

On the other hand, in the last decade, people began to realize that, the derived from Anosov
diffeomorphism does inherit topological hyperbolicity from its isotopy class
(see for example,~\cite{BBI09,Ham13,Pot15,HaP}). This weak hyperbolicity was used
in \cite{HU,Ure12,VY} to deduce measure-theoretical properties for derived from Anosov diffeomorphisms.

By the variation principle, for any invariant probability measure of a diffeomorphism, the metric entropy
is always bounded by the topological entropy. An invariant
probability is called \emph{maximal} if the corresponding metric entropy coincides with the
topological entropy of this diffeomorphism. In other words, the maximal measures are the
measures which are most complicated. It is a well-known fact that every transitive Anosov
diffeomorphism admits a unique maximal measure, and this
maximal measure has exponential decay of correlations for H\"{o}lder continuous observables (see for instance \cite{B}).
It was observed by Ures in \cite{Ure12} \footnote{\label{f.pointwise}In
\cite{Ure12}, the partially hyperbolic diffeomorphism is supposed to be uniform absolute, that is, the parameters
$\lambda_i(x)$, $i=1,\dots, 6$ in the definition of partially hyperbolic diffeomorphism do not depend on $x$.}  and Viana-Yang \cite{VY}  that every diffeomorphism $f\in\cD(A)$ also admits a unique maximal measure. Denote this maximal measure of $f$ by $\nu_f$, in this paper we
are going to prove the following:

\begin{main}\label{main}\footnote{
In \cite{CN}, a different kind of derive fromd Anosov diffeomorphism was considered, where
they assume the existence of Markov partition, some 'good' component where
the center direction is uniformly expanding, and on the 'bad' components
the center direction does not contract too much, that is, the small norm is bounded from below by a
value close to one (condition (5)). With $C^{1+}$ regularity assumption, they proved
similar results. We thank Paulo Varandas for the discussion on this work.}
Suppose $A$ is a three dimensional linear Anosov diffeomorphism over $\TT^3$. Then
for any $C^1$ diffeomorphism $f\in \cD(A)$, its maximal measure $\nu_f$ has exponential decay of correlations for
H\"{o}lder continuous observables: for $0 < \gamma < 1$ there exists some constants $0<\tau<1$ such that for
all $\phi,\psi\in C^\gamma(M)$ there exists $K(\phi,\psi)>0$ satisfying
$$
\mid \int (\phi\circ f^n)\psi d\nu_f - \int \phi d\nu_f \int \psi d\nu_f\mid \leq K(\phi,\psi) \tau^n, \text{ for every }n\geq 1.
$$

\end{main}


We also obtain a large deviation estimate for $C^0$ functions. For any function $\phi$, let $S_n(\phi) = \sum_{k=0}^{n-1}\phi\circ f^k$ be the ergodic sum. We have the following:
\begin{main}\label{C0_deviation}
Suppose $A$ is a three dimensional linear Anosov diffeomorphism over $\TT^3$ with negative center exponent. For every $\phi \in C^0(M)$ with $\nu_f(\phi)=0$ and every $\epsilon>0$ there exists constants $C_\epsilon, c_\epsilon>0$ such that
$$
\nu_f(|S_n(\phi)|>\epsilon n ) \le C_\epsilon e^{-c_\epsilon n}.
$$
\end{main}

Because in Theorem~\ref{main} we can always replace $f$ by its inverse, throughout this paper, we always assume $A$
has negative center exponent.

\subsection{Stretch of proof}

We will use an argument that is similar to \cite{D}, where the coupling
method was used to study the Gibbs-$u$ states of three dimensional $C^{1+}$ partially hyperbolic diffeomorphisms that are:
\begin{itemize}
\item[(i)] \emph{u-convergent \cite[Section 2]{D}},
\item[(ii)] \emph{mostly contracting along the center direction\cite[Section 3]{D}}.
\end{itemize}

First of all, for the maximal measure $\nu_f$, we need to define reference measures on {\bf every} unstable plaque
that are different from the Lebesgue measure, which is explained in Subsection~\ref{ss.local}. More precisely, we deduce this class of measures using the Franks' semiconjugation between the derived from Anosov diffeomorphism $f$ and the linear Anosov map $A$. In principle, this class of measures is only defined on a full measure subset, but these reference measures defined on the unstable plaques are invariant under the holonomy map induced by the center-stable foliation, hence we may define it on the whole manifold.

To apply the coupling argument, we need to show that the properties (i) and (ii) above are satisfied by our reference measures.
The main difficulty in our proof is that, for {\bf every} derived from Anosov diffeomorphisms in $\mathcal{D}(A)$, we need  to deduce the strong measure-theoretical hyperbolicity from the weak topological hyperbolicity (Proposition~\ref{p.coherent}).
In Section~\ref{s.negative}, we show the measures $\nu_f$ have uniform negative center exponent. Moreover, in Section~\ref{s.classification}
we show that $\nu_f$ is u-convergent; indeed we show that the support of $\nu_f$ is a $u$-minimal component (Subsection~\ref{ss.mc}) and that it has mostly contracting center (Subsection~\ref{ss.support}).

The proof of Theorem~\ref{C0_deviation} can be found in Section~\ref{sec_deviation}. The coupling argument for Theorem~\ref{main} is explained through Sections \ref{s.argument} to \ref{s.coupling}.

In the classic theory regarding decay of correlation, the $C^{1+}$ regularity is used for the following three reasons:
\begin{itemize}
\item the transfer operator preserves the space of H\"older functions;

\item the distortion estimation along the unstable plaques;

\item the absolutely continuity of center-stable holonomy assuming the center exponents are all negative (by Pesin theory \cite{P}).
\end{itemize}
In order to make the coupling argument work for $C^1$ diffeomorphisms, first we observe that the Jacobian of $f$ with respect to the reference measures on $u$-plagues is piecewise constant (Proposition~\ref{p.Econtracting}).
Moreover, instead of using the bounded distortion, we use a weak estimate -- Proposition~\ref{p.distortion}, which works for $C^1$ diffeomorphisms. Finally, because the holonomy map induced by the center stable foliation preserves
the reference measures, we may avoid using the Pesin theory.

\nt{{\bf Acknowledgements}. We are grateful to the anonymous referee for a careful revision
of the manuscript.}

\section{Preliminary}
We assume $A: \TT^3\to \TT^3$ to be a linear hyperbolic torus automorphism with eigenvalues $0<\kappa_1<\kappa_2<1<\kappa_3$ and eigenspaces $E_1, E_2, E_3$ respectively, and $f\in\cD(A)$ to be a derived from Anosov diffeomorphism.

We treat $A$ as a partially hyperbolic diffeomorphism with invariant subbundles
$E^{s}_A = E_1$, $E^c_A = E_2$ and $E^{u}_A = E_3$.
Denote by $\omega$ the maximal measure of $A$, it is well-known that, $\omega$ is indeed the
volume measure. We also denote by $\cF_A^{i}$ ($i=s,c,u,cs,cu$)
the linear foliation tangent to the subbundles $E^{s}_A,E^{c}_A,E^{u}_A,E^{cs}_A=E^{s}_A\oplus E^{c}_A,
E^{cu}_A=E^{c}_A\oplus E^{u}_A$ respectively.

\subsection{Dynamical coherence}

By Franks~\cite{Fra70}, there exists a continuous surjective map $h: \TT^3\to \TT^3$
which semiconjugates $f$ to $A$: $h\circ f=A\circ h$. The following proposition shows
that every three dimensional derived from Anosov diffeomorphism admits a weak form of topological hyperbolicity.

\begin{proposition}\label{p.coherent}
$f$ is dynamically coherent, the Franks' semiconjugation $h$ maps the center stable, center, center unstable and unstable leaves of $f$ into the corresponding leaves of $A$. Moreover,
\begin{itemize}
\item[(a)] restricted to each unstable leaf of $f$, $h$ is bijective;

\item[(b)] there is $K>0$ depending only on $f$, such that for every
$x\in \TT^3$, $h^{-1}(x)$,  called the fiber of the semiconjugacy is either a point, or a connected center segment of $f$ with length bounded
by $K$.

\item[(c)] the stable, center, unstable foliation of $f$ are quasi-isometric, that is, there exist $a,b>0$ such that
for any two points $\tilde{x},\tilde{y}$ belonging to the same lifted leaf $\tilde{\cF}^i$ ($i=s,c,u$) in the universal
covering space $\mathbb{R}^3$,
$$d_{\tilde{\cF}^i}(\tilde{x},\tilde{y})<ad(\tilde{x},\tilde{y})+b$$
where $d(\cdot,\cdot)$ is the Euclidean metric on $\RR^3$.

\item[(d)]  the fibers of the semiconjugacy are invariant under unstable holonomy, that is, for any $x,y\in \TT^3$, the unstable foliation of $f$ induces a holonomy map which maps $h^{-1}(x)$ to $h^{-1}(y)$.
\end{itemize}
\end{proposition}
\begin{proof}
It is proven by Potrie~\cite[Theorem A.1]{Pot15} that $f$ is dynamically coherent. Moreover, in~\cite[Theorem 7.10]{Pot15}
he shows that the semiconjugation $h$ maps each center stable leaf of $f$ to a center stable leaf of $A$.
By considering the inverse of $f$, one may show that $h$ also maps the center unstable leaf of $f$
into a center unstable leaf of $A$. Because every center leaf of $f$ belongs to the intersection of
the corresponding center stable leaf and center unstable leaf, by the previous discussion,
$h$ maps every center leaf of $f$ to a center leaf of $A$.

The item (a) is proven in~\cite[Corollary 7.7, Remark 7.8]{Pot15}, and item (b) is
proven by \cite{Ure12} (see also \cite[Proposition 3.1]{VY}). $\cF^c$ being quasi-isometry is proven by
Hammerlindl and Potrie in~\cite[Section 3]{HaP}.  Item (d) is a simple corollary of the fact that
the semiconjugacy maps the center, unstable and center unstable leaves of $f$ into the corresponding leaves of $A$.
\end{proof}

\begin{remark}\label{r.measurepreserving}

By Ledrappier-Walter's formula~\cite{LeW77}, $h_*$ preserves metric entropy. In particular,
$h_*(\nu_f)=\omega$

\end{remark}

The following proposition shows that this topological
hyperbolicity implies that the ergodic measures with high entropy for any derived from Anosov diffeomorphism $f$ and
for the linear Anosov diffeomorphism $A$ are essentially the same.

\begin{proposition}\cite[Theorem 3.6]{VY}\label{p.isomophic}
Let $\mu$ be an ergodic probability measure of $f$ with $h_\mu(f)>-\log \kappa_1$.
Then for $\mu$ almost every $x$, $h^{-1}\circ h(x)=\{x\}$, that is, the map $h_*$ induced by
the semiconjugation $h$ is bijective on the set of ergodic measures with entropy larger than $-\log \kappa_1$.

\end{proposition}

For the further discussion, we also need the following property:

\begin{lemma}\label{l.singleleaf}

For every $x$ belonging to a full measure subset $\Gamma_A$ of $\omega$, $h^{-1}(\cF^{u}_A(x))$ consists of a single
unstable leaf.

\end{lemma}
\begin{proof}
By Proposition~\ref{p.isomophic}, there is a $\nu_f$
full measure subset $\Gamma$ such that for every $y\in \Gamma$, $h^{-1}\circ h(y)=y$.
Let $\Gamma_A=h(\Gamma)$, then by Remark~\ref{r.measurepreserving}, $\omega(\Gamma_A)=1$
(for the measurability of $\Gamma_A$ see~\cite[Corollary 3.4]{VY}).

For every point $x\in \Gamma_A$, by Proposition~\ref{p.coherent}(a), $h^{-1}(\cF^u_A(x))$ is a union of unstable leaves of $f$. Because $h^{-1}(x)$ consists of a unique point, by Proposition~\ref{p.coherent}[(d)], for every $y\in \cF^u_A(x)$,
$h^{-1}(y)$ consists of a single point; hence, $h^{-1}(\cF^{u}_A(x))$ consists of a single
unstable leaf.

\end{proof}

\subsection{Markov partition along $\cF^u$}
In this subsection, we will build a Markov partition along the unstable foliation of $f$, and
consider the disintegration of $\nu$ along this partition, where a partition $\cA$ is said to be Markov
if it is increasing under iterations of $f$ ($f\cA \prec \cA$). We say that a partition $\cA$ is along the unstable foliation if
for almost every $x$, the element of the partition containing $x$ belongs to some unstable leaf. We refer the readers to~\cite{LY2}
for more details (in which it is called partitions subordinate to $W^u$).

We start with a Markov partition $\cM^A=\{M^A_1,\dots, M^A_k\}$ for the linear Anosov map $A$, which
enables us to define a partition $\xi^A$ along the unstable foliation, such that the elements are
the connected components of the intersection of each unstable leaf with $M^A_i$ ($i=1,\dots, k$).
This partition is clearly a Markov partition.

The two partitions $\cM^A$ and $\xi^A$ above induce similar partitions for $f$:
\begin{itemize}
\item $\cM=\{h^{-1}(M^A_i); i=1,\dots, k\}$;
\item $\xi^u=\{\xi^u(x)=(h\mid_{\cF^u(x)})^{-1}(\xi^A(h(x)))\}$.
\end{itemize}
By the Franks' semiconjugation and Proposition~\ref{p.coherent} (a), $\xi^u$ is a Markov partition of $f$ along the unstable foliation given by the intersection of elements of $\cM$ and $\cF^u$.

By Proposition~\ref{p.isomophic}, there is a $\nu_f$ full measure subset $\Gamma$ and a $\omega$ full
measure subset $\Gamma_A$ such that $h\mid_{\Gamma}$ is bijective and preserves the measures, that is,
for any measurable subset $X\subset \Gamma$, $\nu(X)=\omega(h(X))$. By Lemma~\ref{l.singleleaf},
$h^{-1}_*$ maps the conditional measure $\omega^u_{(\cdot)}$ of $\omega$ corresponding to the partition $\xi^A$ to the
conditional measure $\nu^u_{(\cdot)}$ of $\nu$ corresponding to the partition $\xi^u$. That is, for any $x\in \Gamma_A$,
$(h^{-1})_* \omega^u_x=\nu^u_{h^{-1}(x)}$.

\subsection{Local product structure\label{ss.local}}

It is well know that the maximal measure $\omega$ of the linear Anosov diffeomorphism is Lebesgue.
Since $A$ has linear foliations, we have, for $i=1,\dots,k$ and $x^A_i\in M^A_i$,
denote by $\cF^{*}_{A,loc}(x^A_i)$ ($*=cs, u$) the connected component of $\cF^{*}\cap M^A_i$ which contains
$x^A_i$, then  for each $i=1,\dots,k$, there are measure $\omega^{cs}_i$ and $\omega^{u}_i$
supported on $\cF^{cs}_{A,loc}(x_i)$ and $\cF^u_{A,loc}(x_i)$ respectively, such that
\begin{equation}\label{eq.product}
\omega\mid M^A_i=\omega^{cs}_i\times \omega^{u}_i.
\end{equation}

Then \eqref{eq.product} implies that in each foliation chart $M^A_i$, the center-stable holonomy
preserves the conditional measures $\omega^u_{(\cdot)}$.

\begin{definition}\label{d.holonomy}
Denote by $H^{cs}_{x,y}: \xi^u(x)\to \xi^u(y)$ the \emph{center stable holonomy map}. More precisely, for $x,y \in M_i$, $z \in \xi^u(x)$, $H^{cs}_{x,y}(z)$ is the unique point on $\xi^u(y)$ given by $\{H^{cs}_{x,y}(z)\} = \xi^u(y) \cap \cF^{cs}_{loc}(z)$. Here $\cF^{*}_{loc}(z), * = cs, u$ is the connected component of $\cF^* \cap M_i$ that contains $z$.
\end{definition}

For every $1\leq i \leq k$, fix any $x_i\in h^{-1}(x^A_i)$. Recall that by Proposition~\ref{p.coherent},
$h$ preserve the unstable and center-stable foliations, and restrict to each unstable leaf of $f$, $h$
is bijective. Because $h_*(\nu_f)=\omega$, we have that:

\begin{proposition}\label{p.productstructure}

For every $i=1,\dots, k$, there are measures $\nu^{cs}_i$ and $\nu^u_i$ supported on
$\cF^{cs}_{loc}(x_i)$ and $\cF^u_{loc}(x_i)$ such that
$$\nu\mid _{M_i}=\nu^{cs}_i\times \nu^u_i.$$
This implies that for any $y\in \Gamma\cap M_i$, the conditional measure $\nu^u_y=(H^{cs}_{x_i,y})_*(\nu^u_i)$.
In particular, for any $y,z\in \Gamma\cap M_i$, $H^{cs}_{y, z}$ maps the conditional measure $\nu^u_y$ to
$\nu^u_z$ with Jacobian equal to $1$.

\end{proposition}

By the above proposition, we can indeed extend the family of conditional measures $\{\nu^u_x: x \text{ in some full measure subset }\Gamma\}$ to the whole $\TT^3$. More precisely, for every $x\in M_i$, define
$$\nu^u_x=(H^{cs}_{x_i,x})_*\nu^u_i.$$



By the uniqueness of disintegration of $\nu$ ( see~\cite{Ro49} and the survey paper~\cite{CK} for more details on the existence and uniqueness of conditional measures), the Jacobian of $f$ with respect to $\nu^u$ is piecewise constant. To be more precisem, we have:

\begin{lemma}\label{l.Jac}

For every $x\in M$,
$$\frac{df_*(\nu^u_{x})\mid_{\xi^u(f(x))}}{d\nu^u_{f(x)}}(f(x))=\nu^u_x(f^{-1}\xi^u(f(x))).$$

\end{lemma}

\begin{remark}\label{r.lowboundary}
$\nu^u_x(f^{-1}\xi^u(f(x)))$ can only take finitely many possible values. By the discussion above, we have
$$\nu^u_x(f^{-1}\xi^u(f(x)))=\omega^u_{h(x)}(A^{-1}(\xi^A(A(h(x))))).$$
Recall that $\xi^A$ is a partition consists of linear unstable plaques of $A$, and $\omega^u_{(\cdot)}$
is the normalization of Lebesgue measure restricted on the unstable plaque, the latter item of the above
equality can only be one of finitely many positive values.

From now on we take $a_1=\min_{x\in \TT^3} \{\nu^u_x(f^{-1}\xi^u(f(x)))\}>0$.
\end{remark}

\subsection{High iteration\label{ss.iteration}}
In this subsection we claim that, to prove Theorem~\ref{main} for $f$, it suffices
to consider its iteration $f^N$ for any $N>0$.

This is because all iterates of $f$ have the same measure of maximal
entropy. Suppose
Theorem~\ref{main} holds for $f^N$ with $N>0$, that is, for any
$\phi,\psi\in C^\gamma(M)$ there exists $K_N(\phi,\psi)>0$ satisfying
$$
\mid \int (\phi\circ f^{nN})\psi d\nu_f - \int \phi d\nu_f \int \psi d\nu_f\mid \leq K_N(\phi,\psi) \tau^n=(\tau^{1/N})^{nN}, \text{ for every }n\geq 1.
$$

Then for $m=nN+i$ with $1\leq i<N$,
\begin{equation*}
\begin{aligned}
&\mid \int (\phi\circ f^m)\psi d\nu_f - \int \phi d\nu_f \int \psi d\nu_f\mid\\
&=\mid \int (\phi\circ f^i)\circ f^{nN} \psi d\nu_f - \int \phi\circ f^i d\nu_f \int \psi d\nu_f\mid\\
&\leq K_N(\phi\circ f^i,\psi) (\tau^{1/N})^{nN}\\
&=\frac{K_N(\phi\circ f^i,\psi)}{\tau^{i/N}} (\tau^{1/N})^{m}.\\
\end{aligned}
\end{equation*}

Take $K(\phi,\psi)=\max_{i=0,\dots, N-1}\{\frac{K_N(\phi\circ f^i,\psi)}{\tau^{i/N}}\}$, we conclude
the proof of Theorem~\ref{main}.

\section{Negative center exponent\label{s.negative}}

In this section we show that the center Lyapunov exponent of $\nu_f$ is negative, which was
proven in~\cite{Ure12} for every $C^{1+\alpha}$ derived from Anosov diffeomorphism.
We will show that this proposition indeed holds for any $C^1$
derived from Anosov diffeomorphism.

\begin{theorem}\label{t.negativecenter}
Suppose $f$ is a $C^1$ diffeomorphism belongs to $\cD(A)$, then $\lambda^c(\nu_f)\leq \log\kappa_2<0$.

\end{theorem}

\begin{proof}
First, let us observe that three dimensional partially hyperbolic diffeomorphisms are always $C^1$
away from tangencies (see~\cite{W}). Hence, by \cite{LVY}, the metric entropy varies upper-semi continuously with respect to the invariant measures and the diffeomorphisms.

By Remark~\ref{r.measurepreserving}, for any diffeomorphism
$g\in\cD(A)$, the topological entropy is always equal to $h_{top}(A)$. Moreover, by~\cite{Ure12},
$g$ admits a unique maximal measure $\nu_g$. If we take a sequence of $C^2$ diffeomorphisms $g_n \to f$ in the
$C^1$ topology and $\mu$ an accumulation point of $\nu_{g_n}$, we get that
$$h_{\mu}(f) \geq \limsup\limits_n h_{\nu_{g_n}}(g_n) = h_{top}(A).$$
By the uniqueness of the maximal measure of $f$, we have $\mu = \nu_f$.  Hence, we conclude that the
maximal measure $\nu_f$ varies continuously respect to the diffeomorphisms.
Because the center bundle is one dimensional, we get that
$$\lambda^c(\nu_f)=\int \log|df\mid_{E^c}(x)|d\nu_f(x)$$
varies continuously respect to the diffeomorphisms. Therefore, it suffices for us to prove the above theorem for $C^2$ diffeomorphisms.  This was proved by Ures in \cite{Ure12}[Theorem 5.1] for the absolute partially hyperbolic diffeomorphisms which are derived from Anosov. Here we are going to explain how Ures' proof works for any
partially hyperbolic diffeomorphisms: The only place where the absolute partial hyperbolicity
is used in that proof is that the unstable foliation is quasi-isometric, but this does hold for
any partially hyperbolic diffeomorphisms in the same isotopy class of $A$ (Proposition~\ref{p.coherent}[(c)]).

\end{proof}

\section{Classification of $\nu_f$\label{s.classification}}

In this section we are going to build a contracting property along $E^{cs}$ bundle for every element of $\xi^u$ (Proposition~\ref{p.mc}).
The argument depends on the following two properties:
\begin{itemize}
\item the maximal measure is unique;
\item the conditional measure $\nu^u_{\cdot}$ of $\nu_f$ along the partition $\xi^u$
      is invariant under center-stable holonomy map (Proposition~\ref{p.productstructure}).
\end{itemize}
As we have already mentioned, these two properties are consequences of the topological hyperbolicity
(Proposition~\ref{p.coherent}) of the derived from Anosov diffeomorphism.

\subsection{Special probability measure spaces}
In this subsection we are going to introduce a special class of probability measures which
are defined on unstable plaques. A similar definition was used by Dolgopyat in \cite[Section 5]{D}
(with different reference measure) to study physical measures.  Different from \cite{D}[Corollary 6.3],
the proof of uniqueness of invariant measure in the special class of measures used in this paper (Lemma~\ref{l.Eunique}) is much simpler.

Fix $0<\gamma<1$ which denotes the regularity of H\"{o}lder functions. Let $E_1(R)$ be the set of probability measures $l$ such that   for  $\varphi\in C^0(\TT^3)$,
$$l(\varphi)=\int_{\xi^u(x)}\varphi(z) e^{G(z)} d\nu^u_x(z),$$
where  $l(1)=1$ and $|G(z_1)-G(z_2)|\leq R d^{\gamma}(z_1,z_2)$ for any $z_1,z_2\in \xi^u(x)$.  In other words, $E_1(R)$ is the space of probability measures that are absolutely continuous with respect to some $\nu_x^u$, with density $e^{G(z)}$ for some H\"older continous function $G$.

Let $E_2(R)$ be the convex hall of $E_1(R)$ and $E(R)$ the closure of $E_2(R)$ under weak$^*$ topology. The family
$E(R)$ is continuous: $E(R_0)=\bigcap_{R>R_0}E(R)$ (This follows from the fact that $E_1(R_0)=
\bigcap_{R>R_0}E_1(R)$).
Denote by $\cT(l)=l(\varphi\circ f)$ the transfer operator, and $\lambda_5=\max_x \lambda_5(x)$.

\begin{remark}\label{r.topologyofmeasures}
In the following we will consider two different types of convergence in the above measure spaces:
\begin{itemize}
\item the convergence respect to the weak$^*$ topology, which is mainly used in this section;
\item and a more subtle
control on the speed of convergence with respect to the norm $\|\cdot\|_\gamma$ on the space $C^\gamma(\TT^3)$, where
$\|l\|_\gamma$ denotes the norm of $l$ as the element of $(C^\gamma(\TT^3))^*$.
\end{itemize}
Unless otherwise explained, we will be using the first type of convergence.
\end{remark}

\begin{proposition}\label{p.Econtracting}

$\cT: E(R)\to E(Re^{-\lambda_5\gamma})$.
\end{proposition}
\begin{proof}
It suffices to prove that $\cT: E_1(R)\to E_2(Re^{-\lambda_5\gamma})$.

Take $l\in E_1(R)$ such that $\supp(l)\subset \xi^u(x)$ for some point $x\in M$. Then
$$\cT(l)(\varphi)=\int_{\xi^u(x)} e^{G(x)}\varphi(f(z))d\nu^u_x(z)=\int_{f(\xi^u(x))}e^{G\circ f^{-1}(y)}\varphi(y)df_*\nu^u_x(y).$$

Let $f(\xi^u(x))=\bigcup_{i} \xi^u(x_i)$, by the uniqueness of disintegration (Lemma~\ref{l.Jac}), $f_*(\nu^u_x)=\sum c_i \nu^u_{x_i}$,
where $c_i=\nu^u_x(f^{-1}(\xi^u(x_i)))$.
Then $\cT(l)=\sum c_i l_i$ where
$$l_i(\varphi)=\int_{\xi^u(x_i)} e^{G\circ f^{-1}(y)}\varphi(y) d\nu^u_{x_i}(y).$$
Because $|(G\circ f^{-1})(y_1)-(G\circ f^{-1})(y_2)|\leq Re^{-\lambda_5\gamma}d^\gamma (y_1,y_2)$, we have $\cT: E_1(R)\to E_2(Re^{-\lambda_5\gamma})$.
\end{proof}

\begin{lemma}\label{l.Eunique}
$E(0)$ contains a unique invariant probability measure: $\nu_f$.
\end{lemma}
\begin{proof}
First notice that $\nu_f$ is contained in $E(0)$ and is invariant.

Suppose $\mu$ is an invariant probability measure contained in $E(0)$.
By the definition of $E(0)$, the conditional measures of $\mu$ along
the partition $\xi^u$ are $\nu^u_{(\cdot)}$ almost everywhere. We claim that
the disintegration of $h_*(\mu)$ along the partition $\xi^A$ equals to $\omega^u_{(\cdot)}$.
The claim follows from the definition of $\nu^u_{(\cdot)}$: for every $x\in \TT^3$,
$$h_*(\nu^u_x)=\omega^u_{h(x)}.$$

Since among the invariant probability measures of $A$, only $\omega$ admits such disintegration,
the above claim implies that $h_*(\mu)=\omega$. Because $(h^{-1})_*(\omega)=\nu_f$, we complete the
proof.
\end{proof}

\subsection{Mostly contracting center\label{ss.mc}}
 Diffeomorphisms with mostly contracting center were first defined for $C^{1+}$
partially hyperbolic diffeomorphisms by Bonatti-Viana~\cite{BV}
as a technical condition to study physical measures, which, roughly speaking, means that on {\bf every}
unstable plaque, the center-stable bundle $E^{cs}$ has some non-uniform contraction on a positive Lebesgue
measure subset. A somewhat different definition was given by \cite{D}[(6)]; it was shown that such a hypothesis
is important which enables us to apply the coupling argument for physical measures. Although the authors could not
find an available proof, the two definitions on diffeomorphisms with mostly contracting center are equivalent, which can be deduced by the arguments in \cite{D04} and \cite{DVY}. We omit the proof here since the equivalence will not be used in this paper. For more discussion on diffeomorphisms with mostly contracting center, see \cite{DVY}.

In this section, we will show that a contracting property for the center-stable bundle similar to
Dolgopyat's definition (\cite{D}[(6)]) does hold for every $C^1$ derived from Anosov diffeomorphism and for our reference measures. The other kind of contracting property for the center-stable bundle similar to Bonatti-Viana's definition will be
given in the next section.

\begin{remark}
We should notice that, the center exponent of the maximal measure $\nu$ being negative is not sufficient for our proof. Indeed, in order to apply the coupling argument on any pairs of unstable plaques, we will need
the contracting property to hold for {\bf every} unstable plaque.
\end{remark}

\begin{proposition}\label{p.mc}
There is $n_0>0$ and $\alpha_0>0$ such that for any $x\in \TT^3$ and $n\geq n_0$,
\begin{equation}\label{eq.mc}
\int_{\xi^u(x)} \log(df^{n}\mid_{E^c})(y)d\nu^u_{x}(y)\leq - \alpha_0<0.
\end{equation}
\end{proposition}

\begin{proof}
We take $\alpha_0=-\lambda^c(\nu_f)/2>0$. Suppose this proposition is false, then there is
$x_n\in \TT^3$ and $t_n\to \infty$ such that
\begin{equation*}
\int_{\xi^u(x_n)} \log(df^{t_n}\mid_{E^c})(y)d\nu^u_{x_n}(y)\geq -\alpha_0.
\end{equation*}
But any limit of the sequence $\frac{1}{t_n}\sum (f^i)_* (\nu^u_{x_n})$ belongs to $E(0)$ and is an invariant
probability measure, hence coincides to $\nu_f$, this implies that
$$\frac{1}{t_n}\sum (f^i)_* (\nu^u_{x_n})\to \nu_f.$$

Then
\begin{equation*}
\begin{aligned}
\frac{1}{t_n}\int_{\xi^u(x_n)} \log(df^{t_n}\mid_{E^c})(y)d\nu^u_{x_n}(y) &=\frac{1}{t_n}\int_{\xi^u(x_n)} \sum_{i=0}^{t_n-1}\log(df\mid_{E^c})(f^i(y))d\nu^u_{x_n}(y)\\
&=\int_{\xi^u(x_n)}\log (df\mid_{E^c})d\frac{1}{t_n}\sum_{i=0}^{t_n-1} (f^i)_* (\nu^u_{x_n})\\
&\to \lambda^c(\nu_f),\\
\end{aligned}
\end{equation*}
a contradiction.
\end{proof}

\begin{remark}\label{r.immedientlyhyper}
By the discussion in Subsection~\ref{ss.iteration}, for simplicity, we assume $n_0=1$ from now on.
\end{remark}

\subsection{Support of $\nu_f$\label{ss.support}}
In this section we will show a contracting property \eqref{eq.goodcontracting} for our reference measures,
which is similar to the one used by Bonatti-Viana in \cite{BV} with respect to the Lebesgue measure. We are also going to analyze the support of the maximal measure.
\begin{definition}
Suppose $\cF$ is a foliation, we say a compact subset $\Lambda$ is \emph{$\cF$ saturated} if it consists of
union of entire $\cF$ leaves. We say an $\cF$ saturated set $\Lambda$ is \emph{minimal} if every $\cF$ leaf
contained inside $\Lambda$ is dense.
\end{definition}

A hypothesis called \emph{$u$-convergent} was used in \cite{D} to build the coupling technique, which is weaker
than the assumption of minimality of unstable foliation. Later it was shown by \cite{VY13}[Proposision 4.4] that
the support of every physical measure consists of finitely many minimal unstable components. It was shown by \cite{Ure12}
that for any $C^1$ derived from Anosov diffeomorphism $f$, if it is absolutely partially hyperbolic, then the support of
$\nu_f$ is an $\cF$ minimal component. We are going to show that the same argument indeed works for any partially hyperbolic diffeomorphism. 

\begin{proposition}\label{p.hypobolicity}
The support of $\nu_f$ is $u$ saturated. Moreover, $\supp(\nu_f)$
is a minimal $\cF^u$ foliation component. And there are $r_0>0$, $0<a_0<1$ and $C>0$ such that
for every $x\in \TT^3$, there is a set $\Gamma_x\subset \xi^u(x)$ satisfying:
\begin{itemize}
\item[(a)] $\nu_x^u(\Gamma_x)>a_0$;

\item[(b)] for every $y\in \Gamma_x$,
\begin{equation}\label{eq.goodcontracting}
\|df^n\mid_{E^{cs}(y)}\|\leq C e^{n\lambda^c(\nu_f)/2},
\end{equation}
and for any $z\in \cF^{cs}_{r_0}(y)$ and every $n\geq 0$, $d(f^n(y),f^n(z))<Ce^{n\lambda^c(\nu_f)/2}$.
\end{itemize}
\end{proposition}

\begin{remark}\label{r.hyperbolicity}
Replacing $f$ by some power, we may assume that $C=1$. 
Moreover, by changing the metric, we may assume the bundles $E^s,E^c,E^u$ in the partial hyperbolic splitting are orthogonal.
Then by the definition of partial hyperbolicity,
\begin{equation}\label{eq.orthogonal}
\|df^n\mid_{E^{cs}(x)}\|=\|df^n\mid_{E^{c}(x)}\| \text{ for any $x\in \TT^3$ and any $n>0$}.
\end{equation}

\end{remark}
\begin{proof}
Because $\supp(\nu_f)$ is $f$ invariant, to prove that $\supp(\nu_f)$ is $\cF^u$ saturated, it suffices
to show that for any $x\in \supp(\nu_f)$, $\xi^u(x)\subset \supp(\nu_f)$.

For the linear Anosov diffeomorphism $A$, it is well known that for every $z\in \TT^3$, $\supp(\omega^u_{z})=\xi^A(z)$,
this is because, the maximal measure coincides with the Lebesgue measure.
Recall that $h$ is injective between unstable leaves, and for $\nu_f$ typical point $y$,
$h_*(\nu^u_y)=\omega^u_{h(y)}$, we have that $\supp(\nu^u_y)=\xi^u(y)$.

Take $\nu^u$ typical points
$y_n$ converging to $x$, since $\supp(\nu_f)$ is a compact set, this implies that $\xi^u(x)=\lim \xi^u(y_n)\subset \supp(\nu_f)$. Hence $\supp(\nu_f)$ is $\cF^u$ saturated.

 Now we are ready to show that $\supp(\nu_f)$ is indeed a minimal $\cF^u$ component, which
was proven in \cite{Ure12}[Section 6] with the additional assumption that $f$ is absolutely partially hyperbolic.
Now let us show that the same proof still works for any derived from Anosov diffeomorphism. Ures' proof depends on the following two facts: Franks' semiconjugation $h$ maps the center stable, center leaves of $f$ into the corresponding leaves of $A$; and for any point $x\in \TT^3$, $h^{-1}(x)$ is either a point, or a connected center segment of $f$ with uniform bounded length. By
Proposition~\ref{p.coherent}, the above two properties do hold for any derived from Anosov diffeomorphisms. We conclude that $\supp(\nu_f)$ is a minimal $\cF^u$ component.

Let us continue the proof. Take $b_1=e^{\lambda^c(\nu_f)3/4}$. Denote by $\Gamma$ the set of $x$ such that for any $n\geq 0$,
$$\|df^n\mid_{E^{cs}(x)}\|\leq b_1^n.$$
Then $\Gamma$ is a compact set.

\begin{lemma}\label{l.hyperbolictime}
$\nu_f(\Gamma)>0$.
\end{lemma}
\begin{proof}
Take any $\nu_f$ regular point $x$. By Birkhoff ergodic theorem,
\begin{equation}\label{eq.Birkhoff}
\frac{1}{n}\log |df^n\mid_{E^c(x)}|=\frac{1}{n}\sum_{i=0}^{n-1}\log |df\mid_{E^c(f^i(x))}|\to \lambda^c(\nu_f).
\end{equation}

We claim that there is $x^{\prime}=f^j(x)$ for some $j>0$ such that
\begin{equation}\label{eq.hyperbolictime}
|df^n\mid_{E^c(x^\prime)}|<e^{n\lambda^c(\nu_f)3/4} \text{ for any $n\geq 0$ }.
\end{equation}
Suppose this claim is false, then there is $n_0>0$ such that $|df^{n_0}\mid_{E^c(x)}|>e^{n_0\lambda^c(\nu_f)3/4}$.
Also there is $n_1>0$ such that $|df^{n_1}\mid_{E^c(f^{n_0}(x))}|>e^{n_1\lambda^c(\nu_f)3/4}$. $\cdots$.
Then the sequence of positive integers $n_j$ $(j\geq 0)$ satisfies
$$|df^{\sum_{j=0}^{m}n_j}\mid_{E^c(x)}|>e^{(\sum_{j=0}^m n_j)\lambda^c(\nu_f)3/4},$$
a contradiction to \eqref{eq.Birkhoff}.

By Birkhoff ergodic theorem, $\nu_f(\Gamma)>0$.
\end{proof}

Because $b_1< e^{\lambda^c(\nu_f)/2}<1$, by \cite[Lemma 2.7]{ABV}, there is $r_1>0$ such that every $x\in \Gamma$ has
uniform size of stable manifold, which contains $\cF^{cs}_{2r_1}$, more precisely, for any
$y,z\in \cF^{cs}_{2r_1}(x)$ and $n\geq 0$,
$$d(f^n(y),f^n(z))<e^{n\lambda^c(\nu_f)/2}.$$
Moreover, we may assume the bundles $E^s,E^c,E^u$ in the partial hyperbolicity splitting are orthogonal.
Then by the continuation of $\|df\mid_{E^c(x)}\|$ and Remark~\ref{r.hyperbolicity},
$$\|df^n\mid_{E^{cs}(y)}\|=\|df^n\mid_{E^{c}(y)}\| \leq e^{n\lambda^c(\nu_f)/2}.$$

Take a $\nu_f$ generic point $x_0$ such that $\nu^u_{x_0}(\Gamma)=a_1>0$.
There is a neighborhood $V$ of $\xi^u(x_0)$, such that for any point $y\in V$,
the holonomy map $\cH^{cs}_{x_0,y}$ between $\xi^u(x_0)$ and $\xi^u(y)$ satisfies
$d(z,\cH^{cs}_{x_0,y}(z))<r_1$. Because $\xi^u_{\cdot}$ is invariant under center-stable
holonomy map, for any $y\in V$, denote by $\Gamma_y=\cH^{cs}_{x_0,y}(\Gamma\cap \xi^u(x_0))$,
then $\nu^u_y(\Gamma_y)=a_1$.

By the minimality of the unstable foliation $\cF^u$ inside $\supp(\nu_f)$, there is $N>0$, such that for any $z\in \supp(\nu_f)$,
$f^{N}(\xi^u(z))\cap V\neq \emptyset$.
Let $C=(\max \|df\|)^N$, and $r_0=\frac{r_1}{C}$. Then suppose $z^\prime\in f^{N}(\xi^u(z))\cap V$, let
$\Gamma_{z}=f^{-N}(\Gamma_{z^\prime})$. It is easy to see that $\Gamma_{z}$ satisfies (b). It remains
to prove item (a).

Note that $\nu^u_{z}(\Gamma_z)=\nu^u_z(f^{-N}(\xi^u(z^\prime)))\nu^u_{z^\prime}(\Gamma_{z^\prime})\geq \nu^u_z(f^{-N}(\xi^u(z^\prime)))a_1$.

 By Lemma~\ref{l.Jac}, $\nu^u_z(f^{-N}(\xi^u(z^\prime)))=\prod_{j=0}^{N-1} \nu^u_{f^j(z)}(f^{-1}(\xi^u(f^{N-j+1}(z^\prime)))$.
By Remark~\ref{r.lowboundary}, $\nu^u_{f^j(z)}(f^{-1}(\xi^u(f^{N-j+1}(z^\prime)))$ is uniformly bounded from $0$. 
Because $N$ is fixed, we complete the proof of this proposition.

\end{proof}


\section{Large Deviations}\label{sec_deviation}

In this section we prove Theorem~\ref{C0_deviation} in a more general form:

\begin{proposition}\label{p.deviation}
For every $\phi \in C^0(M)$ with $\nu_f(\phi)=0$ and every $\epsilon>0$ there exists constants $C_\epsilon, c_\epsilon>0$ such that for every $l \in E(R)$,
$$
l(|S_n(\phi)|>\epsilon n ) \le C_\epsilon e^{-c_\epsilon n}.
$$
\end{proposition}
\begin{remark}\label{r.iterationforThmB}
 To prove Proposition~\ref{p.deviation}, it suffices to consider an iterations $f^N$ for any $N>0$.

Suppose Proposition~\ref{p.deviation} holds for $f^N$, take $\phi_N=\sum_{i=0}^{N-1}\phi\circ f^i$, for $N\epsilon/2>0$ and
$\phi_N$, by Proposition~\ref{p.deviation}, there are $c,C>0$ such that
$$
l(|S_n(\phi_N)|>\frac{nN\epsilon}{2} ) \le C e^{-n c}.
$$

Take $n_0>0$ such that $\|\phi\|_0\leq  \frac{n_0 \epsilon}{2}$, then for $m>n_0 N$,
write $m=n N+i$ with $1\leq i <N$,
$$l(|S_m(\phi)|>m\epsilon  ) \le l(|S_n(\phi_N)|>\frac{nN\epsilon}{2})  \le Ce^{-n c}\le C e^{c}e^{-m\frac{c}{N}}.$$
Then we can take $c_\vep=\frac{c}{N}$ and $C_\vep=\max\{Ce^{c}, e^{n_0c}\}$ and conclude the proof of Proposition~\ref{p.deviation}.

\end{remark}

In the following, by the discussion in Subsection~\ref{ss.iteration} and Remark~\ref{r.iterationforThmB}, we will prove the above proposition for $f^n$ when $n$ is sufficiently large.


The proof of the above proposition consists of several lemmas. The main idea comes from~\cite{D04}.
 We need to emphasis that, in order to make the argument works for $C^1$ diffeomorphism, instead of using distortion in \cite{D04}[Proposition 4.3], we use a weaker estimation (Proposition~\ref{p.distortion}) which 
can be applied on diffeomorphisms with less regularity.

\begin{lemma}\label{l.largesteplog}
For any continuous function $\phi$ with $\nu_f(\phi)<-\alpha<0$ for some  $\alpha<0$, there is $C_1>0$, such that for any
$n>0$ and $x\in \TT^3$, we have
$$\int_{\xi^u(x)}S_n(\phi)d\nu^u_x\leq -n\alpha/2+C_1.$$
\end{lemma}
\begin{proof}
By an argument similar to Proposition~\ref{p.mc} with $\log(df^n|_{E^c})$ replaced by $S_n(\phi)$,
there is $n_0>0$ such that $\int_{\xi^u(x)}S_{n_0}(\phi)d\nu^u_x\leq -n_0\alpha/2$.
In the following we claim that this lemma is true for $n=kn_0$ (without the constant $C$).

To prove this claim, we use induction. Suppose this lemma holds for $n=n_0,\dots, (k-1)n_0$,
\begin{equation*}
\begin{aligned}
&\int_{\xi^u(x)} S_{kn_0}(\phi)d\nu^u_x=\\
&\int_{\xi^u(x)} S_{n_0}(\phi)\,d\nu_x+\int_{f^{n_0}(\xi^u(x))}S_{(k-1)n_0}(\phi)d(f^{n_0}_*\nu^u_x).
\end{aligned}
\end{equation*}
Let $f^{n_0}(\xi^u(x))=\bigcup \xi^u(x_j)$. Then the second term equals
$$\sum_j c_j \int_{\xi^u(x_j)}S_{(k-1)n_9}(\phi)d\nu^u_{x_j},$$
where $c_j=\nu^u_x(f^{-n_0}(\xi^u(x_j)))$. By induction,
$$\int_{\xi^u(x_j)}S_{(k-1)n_0}(\phi)d\nu^u_{x_j}\leq -(k-1)n_0\alpha/2.$$
Summing over $j$, we complete the proof of the claim.

 To finish the proof of the lemma for every $n$, we write $n=kn_0+m$ with $0\leq m<n_0$. Then we have, by the claim,
$$\int_{\xi^u(x)}S_{n}(\phi)d\nu^u_{x_j}\leq \int_{\xi^u(x)}S_{kn_0}(\phi)d\nu^u_{x_j} + n_0 |\phi| \leq kn_0\alpha/2 + C,$$
with $C = n_0|A|$.
\end{proof}

By the uniform contraction of $f^{-n}$ restricted to $\cF^u$ and continuity of $\phi$ we have:
\begin{proposition}\label{p.distortion}

There is $C>0$ such that for any $\vep>0$, there exists an $n_\vep>0$ such that for any $x\in \TT^3$, $n\geq n_\vep$ and any $y_1,y_2\in f^{-n}(\xi^u(x))$,
$$|S_n(\phi)(y_1)-S_n(\phi)(y_2)|\leq (n+C)\vep.$$

\end{proposition}

As a corollary of Lemma~\ref{l.largesteplog} and Proposition~\ref{p.distortion}:

\begin{corollary}\label{c.discrete}

For any continuous function $\phi$ with $\nu_f(\phi)<-\alpha<0$, there is $\alpha_1>0$ and $n_1 \in \mathbb{N}$ such that for every $n>n_1$ and for any $x\in \TT^3$, write $f^n(\xi^u(x))=\cup_j\xi^u(x_j)$,
we have,
$$\sum_j c_j \max_{f^{-n}(\xi^u(x_j))} S_n(\phi)\leq -\alpha_1<0,$$
where $c_j=\nu^u_{x}(f^{-n}(\xi^u(x_j)))$.

\end{corollary}

By the discussion in Remark~\ref{r.iterationforThmB}, replace $f$ by its iteration $f^n$ for $n$ large,
we think the above corollary works for any $n\geq 1$. With this assumption, we have that:

\begin{lemma}\label{l.notlog}
If $s$ is small enough, there is a constant $\theta_1<1$  such that for every $n\geq 1$,
$$\sum_j c_j\exp\left( s\max_{f^{-n}(\xi^u(x_j))} S_n(\phi)\right)\leq \theta_1.$$

\end{lemma}
\begin{proof}

Consider the function $r_x(s)=\sum_j c_j\exp\left( s\cdot\max_{f^{-n}(\xi^u(x_j))} S_n(\phi)\right)$. Then
$r_x(0)=1$, $\frac{dr_x}{ds}(0)\leq -\alpha_1<0$, and $\left|\frac{d^2r_x}{ds^2}(s)\right|$ is uniformly bounded for any
$x\in \TT^3$ and $s\in [0,1]$. The last inequality comes from the fact that the items inside the sum
is uniformly bounded.

Then the lemma follows immediately from the above observation.
\end{proof}

\begin{corollary}\label{c.largedeviation}
For any $n>0$, and any $x\in \TT^3$, denote by $f^n(\xi^u(x))=\cup_j\xi^u(x_j)$ and
$c_j=\nu^u_x(f^{-n}(\xi^u(x_j)))$, then
\begin{equation}\label{eq.largedeviation}
\sum_j c_j\exp\left( s\max_{f^{-n}(\xi^u(x_j))} S_n(\phi)\right)\leq \theta_1^n.
\end{equation}

\end{corollary}
\begin{proof}
The proof comes from an induction. By the previous lemma, \eqref{eq.largedeviation} is valid for $k=1$. Now assume
that it is correct for all $k\leq n-1$. Let $f(\xi^u(x))=\cup_i \xi^u(y_i)$,
$f^{n-1}(\xi^u(y_i))=\cup_j \xi^u(x_{ij})$, $b_i=\nu^u_x(f^{-1}(\xi^u(y_i)))$ and
$c_{ij}=\nu^u_x(f^{-n}(\xi^u(x_{ij})))$. Then $f^{n}(\xi^u(x))=\cup_{ij} \xi^u(x_{ij})$, and
\begin{equation*}
\begin{aligned}
\sum_{ij} c_{ij}\exp\left( s\max_{f^{-n}(\xi^u(x_{ij}))} S_{n}(\phi)\right)&=\\
\sum_{ij} b_ic_{ij}\exp\left( s\max_{f^{-n}(\xi^u(x_{ij}))} S_{n}(\phi)\right)&\leq\\
\sum_ib_i\exp\left( s\max_{f^{-1}(\xi^u(y_i))} S_1(\phi)\right)\sum_j c_{ij}\exp\left( s\max_{f^{-(n-1)}(\xi^u(x_{ij}))} S_{n-1}(\phi)\right)&\leq\\
\sum_ib_i\exp\left( s\max_{f^{-1}(\xi^u(y_i))} S_1(\phi)\right)\theta_2^{n-1}&\leq\\
\theta_1^{n}.
\end{aligned}
\end{equation*}
\end{proof}





\begin{proof}[Proof of Proposition~\ref{p.deviation}] First we verify the proposition for $l \in E_1(0)$.

Given function $\phi$ with $\nu_f(\phi)=0$, for $\epsilon>0$ define $\tilde{\phi_\epsilon} = \phi - \epsilon$. Then we can apply
Corollary~\ref{c.largedeviation} on $\tilde{\phi_\epsilon}$ and get for some $0<\theta_\epsilon<1$,
$$
\sum_j c_j\exp\left( s\max_{f^{-n}(\xi^u(x_j))} S_n(\tilde{\phi_\epsilon})\right)\leq \theta_\epsilon^n.
$$
This implies that
$$
l(\exp\left( s (S_n(\phi) - n\epsilon)\right)) \le \theta_\epsilon'^n
$$
for every $l \in E(0).$ Now we apply Chebyshev's inequality and obtain
$$
l(S_n(\phi)\ge n\epsilon) \le  \theta_\epsilon'^n.
$$
The same argument applying to the lower bound of $S_n(\phi)$ gives
$$
l(|S_n(\phi)|\ge n\epsilon) \le  \theta_\epsilon''^n,
$$
this finishes the proof for $l \in E(0)$.
Now, given any $l\in E_1(R)$ we write $n = (1-\delta)n + \delta n$ for some $\delta>0$ small. we have
$$
l(|S_n(\phi)|\ge n\epsilon) \le l(|S_{\delta n}(\phi)|\ge\frac{ n\epsilon}{2}) + l(|S_{(1-\delta)n}(\phi)\circ f^{\delta n}|\ge \frac{n\epsilon}{2}).
$$
The first term is $0$ if $\delta$ is chosen small enough. To deal with the second term we assume that
$$
l(\varphi)=\int_{\xi^u(x)}\varphi(z) e^{G(z)} d\nu^u_x(z).
$$
Write $f^{\delta n}\xi(x) = \bigcup_i \xi(x_i)$ and denote by $l_i$ measures on $\xi(x_i)$ with
$$
l_i(\varphi)=\int_{\xi^u(x_i)}\varphi(z) e^{G(f^{-\delta n}z)} d\nu^u_x(z).
$$
Choose some $z_i \in \xi^u(x_i)$ we get that
\begin{align*}
|l_i(\varphi) -\nu^u_{x_i}(\varphi)| =& \int_{\xi^u(x_i)}\varphi(z)(e^{G(f^{-\delta n}z)}-1) d\nu^u_x(z)\\
=&(e^{G(f^{-\delta n}z_i)}-1)\int_{\xi^u(x_i)}\varphi(z) d\nu^u_x(z)\\
\le& C \tilde{\theta}^{\delta n}
\end{align*}
where $C$ and $\tilde{\theta}$ depend on $G$ and $\varphi$ but not on $i$.
As a result
\begin{align*}
&l\left(|S_{(1-\delta)n}(\phi)\circ f^{\delta n}|\ge \frac{n\epsilon}{2}\right) \\
=& \sum_i c_i l_i\left(|S_{(1-\delta)n}(\phi)|\ge \frac{n\epsilon}{2}\right)\\
\le&\sum_i c_i \nu^u_{x_i}\left(|S_{(1-\delta)n}(\phi)|\ge \frac{n\epsilon}{2}\right)+C \tilde{\theta}^{\delta n}\\
\le& C \tilde{\theta}'^{n}.
\end{align*}
\end{proof}

\section{Coupling argument\label{s.argument}}
 This section is similar to \cite{D}[Section 6]. The proof of Theorem~\ref{main} in Subsection~\ref{ss.proofofmain}
is similar to the proof of Theorem I of \cite{D}[Section 10].

 Recall that $E_1(R)$ is the set of measures $l$ such that
\begin{equation}\label{eq.definitionmeasures}
l(\phi)=\int_{\xi^u(x)}\phi(z)e^{G(z)}d\nu^u_x(z),
\end{equation}
where $\phi\in C^0(\TT^3)$, $l(1)=1$ and $|G(z_1)-G(z_2)|\leq R d^\gamma(z_1,z_2)$ for any $z_1,z_2\in \xi^u(x)$.
Let $E_2(R)$ be the convex hall of $E_1(R)$ and $E(R)$ the closure of $E_2(R)$ under weak$^*$ topology.
In order to obtain a subtle control on the speed of convergence, from now on, we will only consider the above measures applying on
the space of functions $C^\gamma(\TT^3)$ and consider the norm $\|\cdot\|_\gamma$. As a result, in the rest of the paper, $\phi$ will always
denote a function in $C^\gamma(\TT^3)$. (See Remark~\ref{r.topologyofmeasures}.)

Let us observe that, although by Lemma~\ref{l.Eunique}, $E(0)$ contains a unique invariant measure, there still exist
plenty of other probability measures which are not necessarily invariant. We want to show that for large $n$ and any $l_1,l_2\in E(R)$, $\cT^n(l_1)$ is exponentially close to $\cT^n(l_2)$  when applied on functions $\phi\in C^\gamma(\TT^3)$.

First we consider the case when $l_1$ and $l_2$ both belong to $E_1(0)$, that
is, $l_i=\nu^u_{x_i}$. Denote by $Y_i=\xi^u(x_i)\times I$ where $I=[0,1]$. Equip $Y_i$ with the measure
$dm_i=d\nu^u_{x_i}\times dt$.

\begin{lemma}\label{l.coupling}
There is a measure preserving map $\tau: Y_1\to Y_2$, a function $R: Y_1\to \mathbb{N}$
and constants $C_1,C_2>0$, $\rho_1<1$, $\rho_2<1$ such that
\begin{itemize}
\item[(A)] If $\tau(y_1,t_1)=(y_2,t_2)$, then for $n\geq R(x_1,t_1)$,
\begin{equation}\label{eq.contracting}
d(f^n(y_1),f^n(y_2))\leq C_1\rho_1^{n-R}.
\end{equation}
\item[(B)] $m_1(R>N)\leq C_2\rho_2^N$.
\end{itemize}
\end{lemma}

The proof of Lemma~\ref{l.coupling} occupies Section~\ref{s.coupling}.

 Recall that $\|l\|_\gamma$ denote the norm of $l$ as an element of $(C^\gamma(\TT^3))^*$.
\begin{corollary}\label{c.converging}
There exist $C_3>0$, $\rho_3<1$ such that for any $n>0$, and any $l_1,l_2\in E(0)$,
$\|\cT^n(l_1-l_2)\|_\gamma\leq C_3\rho_3^n$.
\end{corollary}
\begin{proof}
It suffices to prove for every $l_i\in E_1(0)$. We have
$$(\cT^nl_j)(\varphi)=\int_{Y_j}\varphi(f^n(y_j))dm_j(y_j,t_j).$$
Let $(y_2,t_2)=\tau(y_1,t_1)$. Then
$$\cT^n(l_1-l_2)(\varphi)\leq \int_{Y_1}|\varphi(f^n(y_1)-\varphi(f^n(y_2)))|dm_1(y_1,t_1).$$
Let $Z(n)=\{z: R(z)\leq \frac{n}{2}\}$ then
\begin{equation}
\begin{aligned}
&|\cT^n(l_1-l_2)(\varphi)|\\
&\leq\int_{Z(n)}|\varphi(f^n(y_1))-\varphi(f^n(y_2))|dm_1(y_1,t_1)+2\|\varphi\|_0m_1(Y_1\setminus Z(n))\\
&\leq\|\varphi\|_\gamma\big((C_1\rho_1^{\frac{n}{2}})^\gamma+2C_2\rho_2^{\frac{n}{2}}\big).
\end{aligned}
\end{equation}

\end{proof}

Replace $l_2$ by $\nu_f$, we have that
\begin{corollary}
For any $l\in E(0)$ and any $\varphi\in C^\gamma(M)$, $n>0$,
$$\big|\int \varphi(f^n(x))dl(x)-\nu(\varphi)\big|\leq C_3\rho_3^n\|\varphi\|_\gamma.$$

\end{corollary}

\subsection{Proof of the main results\label{ss.proofofmain}}
To prove Theorem~\ref{main}, we need the following lemma:
\begin{lemma}\label{l.aproximate}
For any $R_0>0$, there is $C_{R_0}$ such that for any $0<R<R_0$ and $l\in E(R)$ there exists $\tilde{l}\in E(0)$
such that $\|l- \tilde{l}\|_0 \leq C_{R_0} \cdot R.$
\end{lemma}
\begin{proof}
It suffices to prove for $l\in E_1(R)$. By the definition of $E_1(R)$, there is a function $G$ such that
$l(\phi)=\int_{\xi^u(x)}\phi(z)e^{G(z)}d\nu^u_x(z)$,
where $\phi\in C^0(\TT^3)$, $l(1)=1$ and $|G(z_1)-G(z_2)|\leq R d^\gamma(z_1,z_2)$ for any $z_1,z_2\in \xi^u(x)$.
We may assume the diameter of every $\xi^u(\cdot)$ is less than one, then
\begin{equation}\label{eq.compare}
\frac{1}{R}\leq \frac{e^{G(y)}}{e^{G(z)}}\leq R \text{  for every $y,z\in \xi^u(x)$}.
\end{equation}
Because $\int e^{G(z)} \nu^u_x =1$, $\min G\mid _{\xi^u(x)}\leq 1 \leq \max G\mid _{\xi^u(x)}$.
By \eqref{eq.compare}, $\frac{1}{R}\leq G\mid _{\xi^u(x)}\leq R$. Thus
$|\int \phi(z) d\nu^u_x(z)-\int \phi(z)e^{G(z)}d\nu^u_x(z)|\leq \|\phi\|_0 \int |1-e^{G(z)}|d\nu^u_x(z)< C_{R_0} R \|\phi\|_0$,
where $C_{R_0}$ is the constant such that $|e^a-1|<C_{R_0} a$ for any $0<a<R_0$.

\end{proof}

\begin{proof}[Proof of Theorem~\ref{main}]
Consider $l\in E(R)$. By Proposition~\ref{p.Econtracting}  and Lemma~\ref{l.aproximate}, there exist $C_R>0$ and $\tilde{l}\in E(0)$ such that
$$\|\cT^{\frac{n}{2}}l- \tilde{l}\|_\gamma \leq C_R \cdot e^{\frac{-\lambda_5 \gamma n}{2}}.$$
Hence, by Corollary~\ref{c.converging}, there is $0<\tau=\max\{e^{\frac{-\lambda_5 \gamma}{2}}, \rho_3\}<1$ such that
$$\|\cT^n l -\nu_f\|_\gamma \leq C_R \cdot e^{\frac{-\lambda_5 \gamma n}{2}}+\|\cT^{\frac{n}{2}}\tilde{l}-\nu_f\|_\gamma\leq C \tau^n.$$

To finish the proof of Theorem~\ref{main}, one only need take $l=\phi \cdot \nu_f$.

\end{proof}

\section{Coupling Algorithm\label{s.coupling}}
In this section we will define $\tau$ and $R$ in Lemma~\ref{l.coupling}.  The arguments of this section are similar to the discussion of \cite{D}[Sections 7,8 and 9]. More precisely, in Subsection \ref{ss.algorithm} we describe the coupling algorithm; in Subsection \ref{ss.firstrun} we describe the first run. These two subsections are parallel to \cite{D}[Sectioon 7]. In Subsection \ref{ss.measurepreserving} we prove (A) of Lemma~\ref{l.coupling}, and in
Subsection \ref{ss.coupling} we prove (B) of Lemma~\ref{l.coupling} in a different argument compares to \cite{D}[Section 9], which is from \cite{Young}.

Although the argument here mainly comes from \cite{D}, the discussion here is simpler. The main differences are: in the first run, we only cut the second coordinate at the step $n_0$ (see Remark~\ref{r.firstrun}); moreover, we will show by construction that after the first run, we define the map between subsets of positive measure (Lemma~\ref{l.positivemeasure}).
The reasons are because our reference measures have some `good' properties:
\begin{itemize}
\item the conditional measure $\nu^u_{\cdot}$ of $\nu_f$ along the partition $\xi^u$
      is invariant under center-stable holonomy map $\cH^{cs}_{\cdot,\cdot}$ (Proposition~\ref{p.productstructure});

\item  the Jacobian of $f$ with respect to $\nu^u$ is piecewise constant:
for every $x\in M$, $\frac{df_*(\nu^u_{x})\mid_{\xi^u(f(x))}}{d\nu^u_{f(x)}}(f(x))=\nu^u_x(f^{-1}\xi^u(f(x)))$ (Lemma~\ref{l.Jac}).
\end{itemize}

Let $\mathcal{Y}$ and $\tilde{\mathcal{Y}}$ be
the set of rectangles $Y=\xi^u(x)\times I$, where $I\subset [0,1]$, with the measure $m=\nu^u_x\times dt$. We write $f(x,t)=(f(x),t)$. For $Y_1\in \mathcal{Y}$ and $Y_2\in \tilde{\mathcal{Y}}$ and the corresponding measures $m_1$ and $m_2$ respectively, such that $m_1(Y_1)=m_2(Y_2)$, we give an algorithm defining $\tau$ and $R$. This algorithm will depend on three positive parameters, $K, \lambda$ and $\vep$.

Applying Corollary~\ref{c.largedeviation} on function $\phi =\log |df|_{E^c}|$, we obtain constants
$s$ and $\theta_1$ such that
for any $n>0$, and any $x\in \TT^3$, denote by $f^n(\xi^u(x))=\cup_j\xi^u(x_j)$ and
$c_j=\nu^u_x(f^{-n}(\xi^u(x_j)))$, then
\begin{equation}\label{eq.exponentialtail}
\sum_j c_j\|df^n\mid E^c\|_{f^{-n}(\xi^u(x_j))}\leq \theta_1^n.
\end{equation}
Let $\lambda>0$ be small enough such that $-\lambda^c(\nu_f)/4>\lambda$ and $e^{-\lambda s}>\theta_1$.

By the uniform contraction of $f^{-1}$ restricted on $\xi^u(\cdot)$, we have

\begin{lemma}\label{l.neighbofhoodofgamma}
There is $K_0>0$ such that for any $n\geq 0$ and $y,z\in f^{-n}(\xi^u(f^n(x)))$,
$$\frac{\|df^n\mid_{E^{cs}(y)}\|}{\|df^n\mid_{E^{cs}(z)}\|}\leq K_0 e^{-n\lambda^c(\nu_f)/4}.$$
\end{lemma}

By Corollary~\ref{c.largedeviation},
take $K>K_0$ large enough such that
\begin{equation}\label{eq.quasihyperbolicset}
q_1=\max_{x} \nu^u_x(U(\xi^u(x)))<1,
\end{equation}
where $U(\xi^u(x))=\{y\in \xi^u(x): \exists n>0 \text{ and } z\in (f^{-n}\xi^u)(y) \text{ such that }\\ \|(df^n\mid_{E^c})(z)\|\geq K e^{-\lambda n}\}$.

\begin{remark}\label{r.gamma}
Recall that by Proposition~\ref{p.hypobolicity} and Remark~\ref{r.hyperbolicity}, there is $a_0>0$ such that for any
point $x\in \TT^3$, there is a set $\Gamma_{x}\subset \xi^u(x)$ with $\xi^u_x(\Gamma_x)>a_0$ and for every $y\in \Gamma_{x}$,
\begin{equation*}
\|df^n\mid_{E^{cs}(y)}\|\leq e^{n\lambda^c(\nu_f)/2}  \text{ for any $n\geq 0$}.
\end{equation*}
Moreover, by Lemma~\ref{l.neighbofhoodofgamma},
$\Gamma_x$ belongs to the complement of $U(\xi^u(x))$.
\end{remark}

There is $\delta>0$ such that if $d(x,y)<\delta$, then $\|(df\mid_{E^{cs}})(y)\|\leq e^{\lambda/2}\|(df\mid_{E^{cs}})(x)\|$. Let $\vep\leq \frac{\delta}{2K}$.

\subsection{Algorithm\label{ss.algorithm}}
The algorithm will work recursively. In the first run, we define the map between $P^\infty_j$ of $Y_j$. For the points where $\tau$ is
not defined, we define a stopping time $s(y)$ such that the set $P_j^n=\{y\in Y_j: s(y)=n\}$ is
of the form  $f^{-n}(\bigcup_kY_{jnk})$ where $Y_{1nk}= \xi^u(x_{1nk})\times I_{1nk}\in \mathcal{Y}$, $Y_{2nk}= \xi^u(x_{2nk})\times I_{2nk}\in \tilde{\mathcal{Y}}$ and $m_1(P^n_1)=m_2(P^n_2)$.

Then we can use our algorithm again to couple $P_1^n$ to $P_2^n$.
We first chop each $Y_{jnk}$ into several pieces along the second coordinate so that the resulting collection $\{\overline{Y}_{jnl}\}$ satisfies $\cup_k Y_{jnk}=\cup_l \overline{Y}_{jnl}$ and $m_1(\overline{Y}_{1nl})=m_2(\overline{Y}_{2nl})$.

Let $f^{-n}(\overline{Y}_{jnl})=U_{jnl}\times I_{jnl}$. Denote $c_{jnl}=\nu^u_{x_j}(U_{jnl})$. Let $\Delta_{jnl}$ be the map
$\Delta_{jnl}(x,t)=(f^n(x),r_{jnl}(t))$ where $r_{jnl}$ is the affine isomorphism between $I_{jnl}$ and $[0,c_{jnl}|I_{jnl}|]$.
We now call our algorithm recursively to produce maps $\tau_{nl}:\Delta(f^{-n}(\overline{Y}_{1nl}))\to \Delta(f^{-n}(\overline{Y}_{2nl}))$
and $R_{nl}: \Delta(f^{-n}(\overline{Y}_{1nl}))\to \mathbb{N}$ satisfying the conditions of Lemma~\ref{l.coupling}.
We set
$$\tau(x,t) =\left\{
  \begin{array}{ll}
    \tau_{\text{first run}}(x,t), & \hbox{if $(x,t)\in P_1^\infty$;} \\
    \Delta_{2nl}^{-1}\circ\tau_{nl}\circ \Delta_{1nl}, & \hbox{if $(x,t)\in f^{-n}(\overline{Y}_{1nl})$.}
  \end{array}
\right.
$$
$$R(x,t) =\left\{
  \begin{array}{ll}
    R_{\text{first run}}(x,t), & \hbox{if $(x,t)\in P_1^\infty$;} \\
    n+R_{nl}(\Delta_{1nl}(x,t)) & \hbox{if $(x,t)\in f^{-n}(\overline{Y}_{1nl})$.}
  \end{array}
\right.
$$
\subsection{First run\label{ss.firstrun}}
Let us now describe the first run of our algorithm.

Because $\supp(\nu_f)$ is a $\cF^u$ minimal component (Proposition \ref{p.hypobolicity}), there exist $n_0$
and points $x_{j1}\in f^{n_0}(\xi^u(x_j))$ for $j=1,2$ such that
$$d_{cs}(y,\cH^{cs}_{x_{11},x_{21}}(y))\leq \vep,$$
for any $y\in \xi^u(x_{11})$, where $d_{cs}$ denotes the distance inside each $cs$ leaf.
Let $\hat{c}_j=\nu^u_{x_j}(f^{-n_0}(\xi^u(x_{j1})))$, $(\overline{t}_1,\overline{t}_2)=(\frac{\hat{c}_2}{\hat{c}_1},1)$
if $\hat{c}_2\leq \hat{c}_1$, and $(\overline{t}_1,\overline{t}_2)=(1,\frac{\hat{c}_1}{\hat{c}_2})$
if $\hat{c}_1\leq \hat{c}_2$. Define $\overline{Y}_j=\xi^u(x_{j1})\times [0,\overline{t}_j]$. Let $s((y,t))=n_0$ for
points of $Y_j\setminus f^{-n_0}(\overline{Y}_j)$.

We now proceed to define $P^n_j$ inductively for $n>n_0$. Let $Q^{n-1}_j=Y_j\setminus \cup_{m=n_0}^{n-1}P_j^m$.
We assume by induction that
$f^{n-1}(Q^{n-1}_j)=\bigcup_k Y_{jk(n-1)}$ where
$$Y_{jk(n-1)}= V_{jk(n-1)}\times [0,\overline{t}_j]= \xi^u(x_{jk(n-1)}))\times [0,\overline{t}_j],$$
\begin{equation}\label{eq.matching}
m_1(f^{-(n-1)}(Y_{1k(n-1)}))=m_2(f^{-(n-1)}(Y_{2k(n-1)})),
\end{equation}
$\cH^{cs}_{x_{1k(n-1)},x_{2k(n-1)}}(\xi^u(x_{1k(n-1)}))=\xi^u(x_{2k(n-1)})$ and
\begin{equation}\label{eq.csdistance}
d(x,\cH^{cs}_{x_{1k(n-1)},x_{2k(n-1)}}(x))\leq r_{n-1} \text{ for any $x \in \xi^u(x_{1k(n-1)})$},
\end{equation}
where $r_n=K\vep e^{-\lambda n/2}$.

Because the partition $\{\xi^u(\cdot)\}$ is increasing, we can write
$$f(\xi^u(x_{jk(n-1)}))=\cup_l \xi^u(x_{jlkn}) \text{ for } j=1,2, $$
such that
$\cH^{cs}_{x_{1lkn},x_{2lkn}}(\xi^u(x_{1lkn}))=\xi^u(x_{2lkn})$.
Let
$$\beta_{lkn}=\|df^{n-n_0}\mid_{E^c({f^{-(n-n_0)}(\xi^u(x_{1lkn}))}}\|.$$
If $\beta_{lkn}>Ke^{-\lambda (n-n_0)}$ let $s(y)=n$ on $f^{-n}(\xi^u_{1lkn})\times [0, \overline{t}_j]$. Otherwise
let $Y_{jlkn}=\xi^u(x_{jlkn})\times [0,\overline{t}_j]$.

To complete the first run, we still need to show the following two lemmas:
\begin{lemma}\label{l.hyperbolictime}
$d(f^n(x),\cH^{cs}_{x_{1jkn},x_{2jkn}}(f^n(x)))\leq r_{n}$ for any $(x,t)\in Y_{1jkn}$,
\end{lemma}
\begin{proof}
This result is a corollary of the following lemma:

\begin{lemma}\cite[Lemma 8.1]{D}\label{l.stablemanifold}
If $x\in M$ and $n>0$ are such that for any $0\leq j< n$, $(df^j\mid_{E^c})(x)\leq K e^{-\lambda j}$ then for any
$0\leq j\leq n$,
$$f^j(\cF^{cs}_{\vep}(x_0))\subset \cF^{cs}_{r_j}(f^j(x_0)),$$
where $r_j=K\vep e^{-\lambda n/2}$.
\end{lemma}

\end{proof}

\begin{lemma}\label{l.volume}
\begin{equation}\label{eq.height}
m_1(f^{-n}(Y_{1jkn}))=m_2(f^{-n}(Y_{2jkn})).
\end{equation}
\end{lemma}
\begin{proof}
Because $\frac{\overline{t}_1}{\overline{t}_2}=\frac{\hat{c}_2}{\hat{c}_1}=
\frac{\nu^u_{x_2}(f^{-n_0}(\xi^u(x_{21})))}{\nu^u_{x_1}(f^{-n_0}(\xi^u(x_{11})))}$ and by definition
$Y_{jlkn}=\xi^u(x_{jlkn})\times [0,\overline{t}_j]$, the lemma follows
from the fact that
$$\frac{\nu^u_{x_1}(f^{-n}(\xi^u(x_{1lkn})))}{\nu^u_{x_2}(f^{-n}(\xi^u(x_{2lkn})))}=
\frac{\nu^u_{x_1}(f^{-n_0}(\xi^u(x_{11})))}{\nu^u_{x_2}(f^{-n_0}(\xi^u(x_{21})))}$$
Observe that by Lemma~\ref{l.Jac}, for $j=1,2$,
$$\frac{\nu^u_{x_j}(f^{-n}(\xi^u(x_{jlkn})))}{\nu^u_{x_j}(f^{-n_0}(\xi^u(x_{j1})))}=\nu^u_{x_{j1}}(f^{-(n-n_0)}(\xi^u(x_{jlkn}))).$$

Then this lemma is a corollary of the fact that $(\cH^{cs}_{x_{11},x_{21}})_*(\nu^u_{x_{11}})=\nu^u_{x_{21}}$
(Proposition \ref{p.productstructure}), and $\cH^{cs}_{x_{11},x_{21}}(f^{-(n-n_0)}(\xi^u(x_{1lkn})))=f^{-(n-n_0)}(\xi^u(x_{2lkn}))$.
\end{proof}

\begin{remark}\label{r.firstrun}
By the above construction, in each run, $\tau$ is not defined on $(x,t)\in Y_1$ for $s(x)=n$
for three reasons:
\begin{itemize}
\item[(a)] $n=n_0$ and $f^{n_0}(x)\notin \xi^u(x_{11})$;

\item[(b)] $n=n_0$ and $f^{n_0}(x)\in \xi^u_{x_{11}}$, but $t>\overline{t}_1$;

\item[(c)] In the step $n>n_0$,
$\max\limits_{y \in f^{-(n-n_0)}(\xi^u(x))}\|df^n\mid_{E^c}(y)\|>Ke^{-\lambda (n-n_0)}$.
\end{itemize}
We only cut the height $I$ to $\overline{t}_1$ in the step $n=n_0$. Thus
$P^\infty_1=Y_1\setminus \cup_n P^n_1$ is a union of vertical intervals of the form $(x,[0,\overline{t}_1])$.
\end{remark}

We define $\tau:P^\infty_1\to P^\infty_2$ such that for any $(x,t)\in P^\infty_1$, denote by
$$y=f^{-n_0}\circ \cH^{cs}_{x_{11},x_{21}}\circ f^{n_0} (x), \text{ then}$$
\begin{equation}\label{eq.definitioncoupling}
\tau((x,t))=(y,\frac{\hat{c}_1}{\hat{c}_2} t) \text{  and   } R(x,t)=n_0.
\end{equation}

We need to verify four things:
\begin{itemize}
\item[(I)] $\tau$ is defined on a set of whole measure in $Y_1$;
\item[(II)] $\tau$ is measure preserving;
\item[(III)] $\tau$ satisfies (A) of Lemma~\ref{l.coupling};
\item[(IV)] $\tau$ satisfies (B) of Lemma~\ref{l.coupling}.
\end{itemize}

\subsection{Measure preserving\label{ss.measurepreserving}}

In this subsection, we verify properties (I), (II) and (III) above.

The property (I) is a sequence of the following lemma:

\begin{lemma}\label{l.positivemeasure}
There is $a>0$ does not depend on $x_1,x_2$ such that $m_1(P^\infty_1)>a$.
\end{lemma}

\begin{proof}
By Proposition~\ref{p.hypobolicity} and Remark~\ref{r.gamma}, there is $a_0>0$ such that for any
point $x\in \TT^3$, there is a set $\Gamma_{x}\subset \xi^u(x)$ with $\xi^u_x(\Gamma_x)>a_0$ and
$\Gamma_x$ belongs to the complement of $U(\xi^u(x))$. We claim that
$$f^{-n_0}(\Gamma_{x_{11}})\times[0,\overline{t}_1]\subset P^\infty_1.$$

Because $\Gamma_{x_{11}}\subset \xi^u(x_{11})$, it does not fit the situation of case (a) of Remark~\ref{r.firstrun} where the points been ruled out in the first run. And by Remark~\ref{r.gamma} and the choice of $K$, the points in $\Gamma_{x_{11}}$ also do not fit the situation of case (c), thus we finish the proof of the claim.

By (a) of Proposition~\ref{p.hypobolicity}, $\nu^u_{x_{11}}(\Gamma_{x_{11}})>a_0$.
Then by Lemma~\ref{l.Jac} and Remark~\ref{r.lowboundary},
$$\nu^u_{x_1}(f^{-n_0}(\Gamma_{x_{11}}))=\nu^u_{x_{11}}(\Gamma_{x_{11}})\prod_{i=0}^{n_0-1} \nu^u_{f^i(x)}(f^{-1}\xi^u(f^{i+1}(x)))>a_0 a_1^{n_0}.$$
Moreover, Remark~\ref{r.lowboundary}, there are only finitely many values of $\hat{c}_j$ ($j=1,2$), so do $\overline{t}_j$, take the minimal
value denote by $t_0$. Let $a= a_0 a_1^{n_0} t_0$, the proof is complete.
\end{proof}

Now let us prove (II):

\begin{lemma}\label{l.measureppreserving}
$\tau\mid P_1^\infty$ is measure preserving.
\end{lemma}
\begin{proof}
One only need to show that the map
$$f^{-n_0}\circ \cH^{cs}_{x_{11},x_{21}}\circ f^{n_0}:  f^{-n_0}(\xi^u(x_{11}))\to f^{-n_0}(\xi^u(x_{21}))$$
has Jacobian
$$\frac{\hat{c}_2}{\hat{c}_1}=\frac{\nu^u_{x_2}(f^{-n_0}(\xi^u(x_{21})))}{\nu^u_{x_2}(f^{-n_0}(\xi^u(x_{11})))}.$$

This comes easily from Lemma~\ref{l.Jac} and
because $\nu^u_{\cdot}$ is preserved by the center-stable holonomy:
$$(\cH^{cs}_{x_{11},x_{21}})_*\nu^u_{x_{11}}=\nu^u_{x_{21}}.$$

\end{proof}

Moreover, (III) comes directly from the construction (see \eqref{eq.csdistance}),
if we take $\rho_1=K\vep e^{-\lambda/2}$.

\subsection{Coupling time\label{ss.coupling}}
Here we prove IV: (B) of Lemma~\ref{l.coupling}. As a summary of the above discussion, we have:

\begin{lemma}\label{l.couplingcontrol}

There are constants $q,C_0>0,\rho_0<1$ such that for any pair $Y_1,Y_2$
\begin{itemize}
\item[(H1)] $m_1(P_1^\infty) \geq a$;
\item[(H2)] $m_1(P^n_1)\leq C_0\rho_0^n$.
\end{itemize}
\end{lemma}
\begin{proof}
(H1) is exactly Lemma~\ref{l.positivemeasure}.

To prove (H2), observe that for $n>n_0$, the only reason for $(y,t)$ to belong to $P_1^n$
is that $\max\limits_{z \in f^{-(n-n_0)}(\xi^u(f^n(y)))}|df^{n-n_0}\mid_{E^c}(z)|>Ke^{-\lambda n}$.
So by the choice of $\lambda$, the measure of such points is exponentially small by
Corollary~\ref{c.largedeviation} and \eqref{eq.exponentialtail}.

\end{proof}

Now represent $R(y)=\sum_{j=1}^{k(y)}s_j(y)$, where $s_j(x)$ is the stopping time of the $j$th run of our algorithm.
And $T_k$ be the set where $\tau$ is not defined after $k$ runs of our algorithm and $U_k=T_{k-1}\setminus T_k$.
Then $m_1(R=n)=\sum_{i\leq \text{[}\delta n\text{]}}m_1(U_i)+\sum_{i>\text{[}\delta n\text{]}}m_1(U_i)=I+II$ for some $\delta$ small

Then $II<(1-a)^{\delta n}$ by (H1) of Lemma~\ref{l.couplingcontrol}. To estimate the first term, we only need to consider
\begin{equation}
\begin{aligned}
m_1(\{R=n\}\cap U_i)&=\sum_{(k_1,\dots,k_i); \sum k_j=n}m_1(\{s_i=k_i\})\\
&\leq \sum_{(k_1,\dots,k_i);\sum k_j=n}(\prod_{j=1}^i C_0 \rho_0^{k_j})\\
&\leq \left(
        \begin{array}{c}
          n \\
          i \\
        \end{array}
      \right)
C_0^i \rho_0^n.\\
\end{aligned}
\end{equation}
Because $\left(
        \begin{array}{c}
          n \\
          \text{[}n\delta\text{]} \\
        \end{array}
      \right)\approx e^{\vep n}$
for some $\vep=\vep(\delta)$ where $\vep(\delta)\to 0$ as $\delta\to 0$.
Choosing $\delta$ small enough such that
$$e^{\vep(\delta)}C_0^\delta \rho_0 = \rho^{\prime}<1.$$
Then the first item is $\leq \text{[}n\delta\text{]} \rho^{\prime n}$.

We complete the proof of (B) in Lemma~\ref{l.coupling}.


\begin{thebibliography}{10}
\bibitem{ABV}
J. F.~Alves, C.~Bonatti, M.~Viana,
\newblock SRB measures for partially hyperbolic systems whose central direction is mostly expanding,
\newblock {\em Invent. Math}. 140: 351--398, 2000

\bibitem{An67}
D.~V. Anosov.
\newblock Geodesic flows on closed {R}iemannian manifolds of negative
  curvature.
\newblock {\em Proc. Steklov Math. Inst.}, 90:1--235, 1967.

\bibitem{AS67}
D.~V. Anosov and Ya.~G. Sinai.
\newblock Certain smooth ergodic systems.
\newblock {\em Russian Math. Surveys}, 22:103--167, 1967.

\bibitem{Beyond}
C.~Bonatti, L.~J. D{\'{\i}}az, and M.~Viana.
\newblock {\em Dynamics beyond uniform hyperbolicity}, volume 102 of {\em
  Encyclopaedia of Mathematical Sciences}.
\newblock Springer-Verlag, 2005.

\bibitem{BV}
C. Bonatti and M. Viana.
\newblock { SRB measures for partially hyperbolic systems whose central direction is mostly contracting.}
\newblock {\em Israel J. Math.}, 115:157--193, 2000.

\bibitem{B}
R.~Bowen
\newblock Equilibrium States and the Ergodic Theory of Anosov Diffeomorphisms.
\newblock {Lect. Notes in Math}. 470, Springer, 1975.

\bibitem{BBI09}
M.~Brin, D.~Burago, and S.~Ivanov.
\newblock Dynamical coherence of partially hyperbolic diffeomorphisms of the
  3-torus.
\newblock {\em J. Mod. Dyn.}, 3:1--11, 2009.

\bibitem{BP74}
M.~Brin and Ya. Pesin.
\newblock Partially hyperbolic dynamical systems.
\newblock {\em Izv. Acad. Nauk. SSSR}, 1:177--212, 1974.

\bibitem{BW10}
K.~Burns and A.~Wilkinson.
\newblock On the ergodicity of partially hyperbolic systems.
\newblock {\em Annals of Math.}, 171:451--489, 2010.

\bibitem{CN}
A.~Castro and T.~Nascimento.
\newblock Statistical properties of the maximal entropy measure for partially hyperbolic attractors.
\newblock {\em Ergodic Theory and Dynamical Systems} 2016.


\bibitem{CK}
V.~Climenhaga and A.~Katok.
\newblock Measure theory through dynamical eyes.
\newblock Available at arxiv.org/abs/1208.4550.


\bibitem{C}
K.~Cogswell.
\newblock Entropy and volume growth.
\newblock {\em Ergodic Theory Dynam. Systems}, 20:77--84, 2000.

\bibitem{D}
D. Dolgopyat.
\newblock On dynamics of mostly contracting diffeomorphisms.
\newblock {\em Comm. Math. Phys,} 213:181--201, 2000.

\bibitem{D04}
D. Dolgopyat.
\newblock Limit theorems for partially hyperbolic systems.
\newblock {\em Tran. Amer. Math. Soc.,} 356:1637--1689, 2004.


\bibitem{DVY}
D. Dolgopyat, M. Viana and J. Yang.
\newblock Geometric and measure-theoretical structures of maps with mostly contracting center.
\newblock {\em Comm. Math. Physics,} 341:991--1014, 2016.

\bibitem{FPS14}
T.~Fisher, R.~Potrie, and M.~Sambarino.
\newblock Dynamical coherence of partially hyperbolic diffeomorphisms of tori
  isotopic to {A}nosov.
\newblock {\em Math. Z.}, 278:149--168, 2014.

\bibitem{Fra70}
J.~Franks.
\newblock Anosov diffeomorphisms.
\newblock In {\em Global {A}nalysis ({P}roc. {S}ympos. {P}ure {M}ath., {V}ol.
  {XIV}, {B}erkeley, {C}alif., 1968)}, pages 61--93. Amer. Math. Soc., 1970.

\bibitem{G}
S.~Gan
\newblock A generalized shadowing lemma.
\newblock {\em Discrete Contin. Dyn. Syst.}, 8, no. 3, 627--632, 2002.

\bibitem{Ham13}
A.~Hammerlindl.
\newblock Leaf conjugacies on the torus.
\newblock {\em Ergodic Theory Dynam. Systems}, 33:896--933, 2013.

\bibitem{HaP}
A.~Hammerlindl and R.~Potrie.
\newblock Pointwise partial hyperbolicity in three-dimensional nilmanifolds.
\newblock {\em J. Lond. Math. Soc.}, 89:853--875, 2014.

\bibitem{HU}
A.~Hammerlindl and R.~Ures.
\newblock Ergodicity and partial hyperbolicity on the 3-torus.
\newblock {\em Commun. Contemp. Math.}, 16:1350038, 22, 2014.

\bibitem{HHUcoh}
F.~Rodriguez Hertz, M.~A.~Rodriguez Hertz, and R.~Ures.
\newblock A non-dynamically coherent example on $\mathbb{T}^3$.


\bibitem{HPS70}
M.~Hirsch, C.~Pugh, and M.~Shub.
\newblock Invariant manifolds.
\newblock {\em Bull. Amer. Math. Soc.}, 76:1015--1019, 1970.

\bibitem{HPS77}
M.~Hirsch, C.~Pugh, and M.~Shub.
\newblock {\em Invariant manifolds}, volume 583 of {\em Lect. Notes in Math.}
\newblock Springer Verlag, 1977.

\bibitem{HSX}
Y.~Hua, R.~Saghin, Z.~Xia,
\newblock Topological entropy and partially hyperbolic diffeomorphisms.
\newblock {\em Ergodic Theory Dynam. Systems} 28 no.3, 843--862, 2008.

\bibitem{LeW77}
F.~Ledrappier and P.~Walters.
\newblock A relativised variational principle for continuous transformations.
\newblock {\em J. London Math. Soc.}, 16:568--576, 1977.

\bibitem{LY2}
F.~Ledrappier and L.-S.~Young.
\newblock The metric entropy of diffeomorphisms. II. Relations between entropy, exponents and dimension.
\newblock {\em Ann. of Math.}, 122:540--574, 1985.

\bibitem{LVY}
G.~Liao, M.~Viana, and J.~Yang.
\newblock The entropy conjecture for diffeomorphisms away from tangencies.
\newblock {\em Journal of the E.M.S.},Volume 15, Issue 6, pp. 2043--2060, 2013.

\bibitem{M}
R. Ma\~{n}\'{e}.
\newblock Contributions to the stability conjecture.
\newblock {\em Topology}, 17:383--396, 1978.

\bibitem{P}
Y. Pesin.
\newblock Characteristic exponents and smooth ergodic theory.
\newblock {\em Russian Math. Surveys}, 32:55-?114, 1977.


\bibitem{PT}
G.~Ponce and A. Tahzibi.
\newblock Central Lyapunov exponent of partially hyperbolic diffeomorphisms of $\TT^3$.
\newblock {\em Proc. Amer. Math. Soc.}, 142:3193--3205, 2014.

\bibitem{Pot15}
R.~Potrie.
\newblock Partial hyperbolicity and foliations in {$\Bbb{T}^3$}.
\newblock {\em J. Mod. Dyn.}, 9:81--121, 2015.

\bibitem{PSh72}
C.~Pugh and M.~Shub.
\newblock Ergodicity of {A}nosov actions.
\newblock {\em Invent. Math.}, 15:1--23, 1972.

\bibitem{PSh89}
C.~Pugh and M.~Shub.
\newblock Ergodic attractors.
\newblock {\em Trans. Amer. Math. Soc.}, 312:1--54, 1989.

\bibitem{Ro49}
V.~A.~Rokhlin.
\newblock Selected topics from the metric theory of dynamical systems.
\newblock {\em Uspekhi Mat. Nauk}, 4:57--125, 1949. 

\bibitem{SX09}
R.~Saghin and Z.~Xia.
\newblock Geometric expansion, {L}yapunov exponents and foliations.
\newblock {\em Ann. Inst. H. Poincar\'e Anal. Non Lin\'eaire}, 26:689--704,
  2009.


\bibitem{Ure12}
R.~Ures.
\newblock Intrinsic ergodicity of partially hyperbolic diffeomorphisms with a
  hyperbolic linear part.
\newblock {\em Proc. Amer. Math. Soc.}, 140:1973--1985, 2012.

\bibitem{UVY}
R.~Ures, M.~Viana and J.~Yang
\newblock Maximal measures of diffeomorphisms with circle fiber bundle.
\newblock Preprint.

\bibitem{VY13}
M.~Viana and J.~Yang
\newblock Physical measures and absolute continuity for one-dimensional center direction.
\newblock {\em Ann. Inst. H. Poincar\'e Anal. Non Lin\'eaire,} 30:845--877, 2013.

\bibitem{VY}
M.~Viana and J.~Yang
\newblock Measure-theoretical properties of center foliations.
\newblock www.arxiv.org.

\bibitem{W}
L.~Wen.
\newblock Generic diffeomorphisms away from homoclinic tangencies and heterodimensional cycles.
\newblock {\em Bull. Braz. Math. Soc.}, 35:419--452, 2004.

\bibitem{Yang16}
J.~Yang.
\newblock Entropy along expanding foliations.
\newblock www.arxiv.org.

\bibitem{Young}  L-S.~Young.
\newblock Recurrence time and rate of mixing.
\newblock {\em Israel J. Math.}, 110:153--188, 1999.


\end{thebibliography}
\end{document}